\documentclass{article}

\usepackage[12pt]{extsizes}
\usepackage{cmap}
\usepackage[utf8]{inputenc}
\usepackage[T2A]{fontenc}
\usepackage[english]{babel}
\usepackage{amsmath}
\usepackage{amsthm}
\usepackage{amssymb}
\usepackage{amsfonts}
\usepackage{enumerate}
\usepackage[affil-it]{authblk}

\newcommand{\RN}{\mathbb{R}} 
\renewcommand{\C}{\mathbb{C}} 
\newcommand{\CN}{\mathbb{C}}
\newcommand{\ZN}{\mathbb{Z}} 

\newcommand{\eps}{\ensuremath\varepsilon}
\newcommand{\intl}{\int\limits}

\newcommand{\ovl}{\overline}

\newcommand{\mc}{\mathcal}
\newcommand{\mb}{\mathbf}

\newcommand{\Bun}{\operatorname{Bun}}
\newcommand{\Pone}{{\mathbb{P}^1}}

\renewcommand{\th}{\vartheta}
\newcommand{\dd}[1]{\frac{\partial}{\partial #1}}
\newcommand{\abs}[1]{\lvert #1\rvert}
\newcommand{\norm}[1]{\| #1\|}
\newcommand{\PGL}{\operatorname{PGL}}
\newcommand{\mby}{\mb{y}}
\newcommand{\mbz}{\mb{z}}

\begin{document}
\newtheorem{thr}{Theorem}[section]
\newtheorem*{thr*}{Theorem}
\newtheorem{lem}[thr]{Lemma}
\newtheorem*{lem*}{Lemma}
\newtheorem{cor}[thr]{Corollary}
\newtheorem*{cor*}{Corollary}
\newtheorem{prop}[thr]{Proposition}
\newtheorem*{prop*}{Proposition}
\newtheorem{stat}[thr]{Statement}
\newtheorem*{stat*}{Statement}
\theoremstyle{definition}
\newtheorem{defn}[thr]{Definition}
\theoremstyle{remark}
\newtheorem{rem}[thr]{Remark}
\newtheorem*{rem*}{Remark}
\newtheorem{example}[thr]{Example}
\author{Daniil Klyuev}

\title{Analytic Langlands correspondence for $\PGL_2(\C)$ on a genus one curve with parabolic structures}
\maketitle
\abstract{Analytic Langlands correspondence was proposed by Etingof, Frenkel and Kazhdan. On one side of this correspondence there are certain operators on $L^2(\Bun_G)$, called Hecke operators, where $\Bun_G$ is the variety of stable $G$-bundles on $X$ and $L^2(\Bun_G)$ is a Hilbert space of square-integrable half-densities. The compactness conjecture says that Hecke operators are bounded and, moreover, compact. In~\cite{EFK3} Etingof, Frenkel and Kazhdan prove this and other conjectures in the case of $G=\PGL_2$ and $X=\Pone$ with parabolic structures. We investigate the case of $G=\PGL_2$ and genus one curve over complex numbers with parabolic structures. We obtain an explicit formula for Hecke operators and prove the compactness conjecture in this case.}
\section{Introduction}
Analytic Langlands correspondence was developed in~\cite{EFK1}, \cite{EFK2}, \cite{EFK3}, \cite{EFK4}, motivated in part by the works~\cite{BK}, \cite{Ko}, \cite{La}, \cite{Te}. The general (conjectural) definition of Hecke operators was given in~\cite{EFK2}. The paper~\cite{EFK3} proves almost all of the conjectures of~\cite{EFK1},~\cite{EFK2} in the case when $G=\PGL_2$, $X=\Pone$. Our plan is to do the same when $F=\C$, $G=\PGL_2$ and $X$ has genus one.

Let $X$ be a smooth projective curve over $\C$, $G$ be a semisimple complex Lie group. We choose a finite subset of points $S\subset X$. A parabolic $G$-bundle is a principal $G$-bundle with a choice of a Borel reduction at each point $t\in S$. In our case $G=\PGL_2$ this means the following. A $G$-bundle is a rank two vector bundle $E$ with equivalence relations $E\sim E\otimes L$ for a line bundle $L$. A choice of Borel reduction at $t$ is the choice of a line in the fiber $E_t$.

Consider the moduli space $\Bun_G^{\circ}$ of stable parabolic $G$-bundles. We will define the Hilbert space $\mc{H}=L^2(\Bun_G^{\circ})$ of square-integrable half-densities below. Hecke operators are a family of commuting bounded operators acting on $\mc{H}$. The main goal of the analytic Langlands correspondence is to describe the joint spectrum of Hecke operators.

One of the main conjectures of analytic Langlands correspondence is the compactness conjecture. It states that Hecke operators are compact (hence, bounded) operators on $L^2(\Bun_G^{\circ})$. We prove the compactness conjecture in our case. 


The structure of the paper is as follows.

We start with describing a birational parametrization of $\Bun_0$, the component of $\Bun_G^{\circ}$ that consists of bundles of even degree. Let $S=\{t_0,\ldots,t_m\}$. Each parabolic point $t_i$ gives rise to a birational involution $S_i$ on $\Bun_G^{\circ}$ that interchanges $\Bun_0$ and $\Bun_1$. We use $S_0$ to identify $\Bun_0$ and $\Bun_1$. Let $HM_{x,l}$ be the Hecke modification at $x$ with parameter $l$ defined below. Instead of computing the action of $HM_{x,l}\colon \Bun_0\to \Bun_1$ we compute the action of $S_0 \circ HM_{x,l}$ on $\Bun_0$. 

We use $S_0 \circ HM_{x,l}$ to define Hecke operators $H_x$. We can do that since the actual Hecke operators look like $\begin{pmatrix}
0 & H_x\\
H_x & 0
\end{pmatrix}$, so their eigenvectors and eigenvalues are in two-to-one correspondence with the eigenvectors and eigenvalues of $H_x$, namely an eigenvalue $\beta$ corresponds to $\pm\beta$. We use the formula for $S_0 \circ HM_{x,l}$ to compute the action of Hecke operators on $L^2(\Bun_0)$. We prove that the Hecke operator for one parabolic point is unitarily equivalent to the Hecke operator for four parabolic points on $\Pone$. Using this, we prove that Hecke operators are bounded for any number of points. 

Similarly to~\cite{EFK3}, the Hecke operator can be written as \begin{equation}
\label{EqGroupAndKernelHeckeOperatorIntroduction}
(H_xf)(p,y_1,\ldots,y_m)=\int_{\C}\big(\rho(g_{x,p,q})f\big)(q,y_1,\ldots,y_m)K(p,q,x)dq,
\end{equation} where $K(p,q,x)=\mathbf{K}(\wp(p),\wp(q),\wp(x))$ and $\mathbf{K}$ coincides with the kernel for the Hecke operator on $\Pone$ with four parabolic points, $g_{x,p,q}$ is an element of $G=\PGL_2^m$, $\rho$ is the standard representation of $G$ on $L^2(\Pone)^m$, and the action of $G$ is in the variables $y_1,\ldots,y_m$. 

 Compactness in~\cite{EFK3} was deduced from the following corollary of the Harish-Chandra theorem: for $f\in L^1(G)$ the operator $K_f=\int \rho(g)f(g)dg$ is compact. We also deduce compactness from this result.

Finally, we check that Hecke operators are self-adjoint and commute with each other.



\subsection{Acknowledgments}
I am grateful to Pavel Etingof for formulation of the problem, stimulating discussions and for helpful remarks on the previous versions of this paper. I would like to thank Davide Gaiotto for telling me the formula for an intertwining kernel in the case of genus one curve.
\section{Main definitions}
The general definition of Hecke operators uses the variety $\Bun_G^{\circ}$ of stable bundles. The definition of a stable quasiparabolic bundle is given, for example, in~\cite{EFK3}, Definition 2.1. We are interested in a birational parametrization of this variety, so we do not need a precise definition of a stable bundle here. Note that $\Bun_G^{\circ}$ is a union of $\Bun_0$ and $\Bun_1$ corresponding to stable bundles with determinant of even and odd degree respectively. As we will explain below, we can use one of the parabolic points to identify $\Bun_0$ and $\Bun_1$. Hence we are interested in a birational parametrization of $\Bun_0$.



Choose a point $0\in X$ for convenience.

Let us say that a vector bundle on $X$ is semisimple if it is isomorphic to a direct sum of line bundles. Note that semisimple rank $2$ bundles form an open subset of $\Bun_0$. Tensoring by a line bundle, we can assume that the determinant bundle $\det E$ is trivial, hence $E=\mc{L}_p\oplus \mc{L}_{-p}$ for some $p\in X$. Here $\mc{L}_p=\mc{O}(p-0)$ corresponds to the Weil divisor $[p]-[0]$.

We can describe $\mc{L}_p$ as follows. Let $X=\CN/\Lambda$, where $\Lambda$ is a lattice generated by $1,\tau$. Choose a lift of $p$ to $\CN$ and denote it also by $p$. Let $U$ be a $\Lambda$-invariant open subset of $\CN$. Then sections of $\mc{L}_p$ over $U$ are given by functions $f$ holomorphic on $U$ such that $f(z+1)=f(z)$, $f(z+\tau)=e^{2\pi i p}f(z)$.

Hence we will write sections of $E$ as pairs $\begin{pmatrix}
f_+ \\ f_-
\end{pmatrix}$ such that $f_{\eps}(z+1)=f_{\eps}(z)$, $f_{\eps}(z+\tau)=e^{2\pi i \eps p}f_{\eps}(z)$ for $\eps=\pm$. 

Recall that the auxiliary Jacobi theta function $\th(z)=\th_{11}(z)$ has simple zeroes at $\Lambda$ and  satisfies $\th(z+1)=-\th(z)$, $\th(z+\tau)=-e^{-\pi i (\tau+2z)}\th(z)$. Hence for any $a\in \CN$ the function $f=f_a(z)=\frac{\th(z-a)}{\th(z)}$ has a simple zero at $a$, a simple pole at $0$ and satisfies $f(z+1)=f(z)$, $f(z+\tau)=e^{2\pi i a}f(z)$. Hence $f_p$ is a rational section of $\mc{L}_p$ with divisor of zeroes $[p]-[0]$.

Let $t_0,\ldots,t_m$ be the marked points on $X$. Let us choose a lift of $t_0,\ldots,t_m$ to $\CN$, which we will also denote $t_0,\ldots,t_m$. We can assume that $t_0=0$.

We will parametrize parabolic lines $l_0,\ldots,l_m$ at $t_0,\ldots,t_m$ as follows. For $i=0,\ldots,m$ we find $y_i\in \Pone$ such that $l_i$ is spanned by $\begin{pmatrix}
y_i\\1
\end{pmatrix}\in E_{t_i}$. We will consider a dense open subset of $\Bun_0$ consisting of parabolic bundles with $y_i\neq 0,\infty$. As a vector bundle, $E$ has a group of automorphisms $\CN^{\times}\times \CN^{\times}$. The action of $(a,b)$ sends $y_i$ to $ab^{-1}y_i$. Hence we can assume that $y_0=1$. We get the following
\begin{prop}
$\Bun_0$ is birationally equivalent to \[(\CN\times (\Pone)^m)/(C_2\ltimes \ZN^2),\] where $C_2$ is a group of order two. The nontrivial element of $C_2$ acts on $\ZN^2$ as $(a,b)\mapsto (-a,-b)$ and on $\CN\times(\Pone)^m$ as $(p,y_1,\ldots,y_m)\mapsto (-p,y_1^{-1},\ldots,y_m^{-1})$. The action of $(a,b)\in \ZN$ sends $(p,y_1,\ldots,y_m)$ to $(p+\tfrac12a+\tfrac12b\tau,e^{2\pi i b t_1}y_1,\ldots, e^{2\pi i b t_m}y_m)$.
\end{prop}
\begin{proof}
From the discussion above we see that for any semisimple parabolic $\PGL_2$-bundle $E$ we can find $(p,y_1,\ldots,y_m)$ such that $E$ is isomorphic to $\mc{L}_p\oplus \mc{L}_{-p}$ with the parabolic structure as above. Two bundles $\mc{L}_p\oplus\mc{L}_{-p}$ and $\mc{L}_q\oplus \mc{L}_{-q}$ are isomorphic to each other if and only if $2(p\pm q)\in \Lambda$ for some choice of sign. 

Note that $(p,y_1,\ldots,y_m)$ and $(-p,y_1^{-1},\ldots,y_m^{-1})$ describe the same parabolic bundle. Hence we can assume that $2(p-q)\in \Lambda$.

Let $q-p=a+b\tau$ with $a,b\in \tfrac12\ZN$. We have $(\mc{L}_p\oplus\mc{L}_{-p})\otimes \mc{L}_{q-p}\cong \mc{L}_q\oplus \mc{L}_{2q-p}\cong \mc{L}_q\oplus \mc{L}_{-q}$. The first isomorphism sends $(f,g)\otimes h$ to $(fh,gh)$, the second sends $(f,g)$ to $(f,ge^{-4\pi i b z})$, hence the line $(1:1)$ at zero stays the same and the line $(y_i:1)$ at $t_i$ goes to the line $e^{4\pi i b z}y_i:1$.
\end{proof}

We recall the notion of a Hecke modification. For $x\in X\setminus S$ and a line $l\subset E_x$ define the Hecke modification of $E$ to be the vector bundle whose sections are the rational sections of $E$ with at most a first-order pole at $x$ with residue that belongs to $l$. This operation can make a stable bundle unstable, but it is defined on a dense open subset of $\Bun_G$.

Note that $(HM_{x,l}E)_y=E_y$ for all $y\neq x$ and we have the short exact sequence of vector spaces $0\to E_x/l\to (HM_{x,l}E)_x\to l\to 0$: the second nontrivial map is taking residue, the first nontrivial map comes from the tautological map $E\to HM_{x,l}E$. Hence we can define $HM_{x,l}$ for parabolic bundles in the case when $x\notin S$ choosing the same line at each $t\in S$. For $t_i\in S$ we can also define $HM_{t_i}=HM_{t_i,l_i}$: the parabolic line at $t_j\neq t_i$ is the same, the parabolic line of $HM_{t_i}E$ at $t_i$ is $E_{t_i}/l_i$. Note that 

\begin{equation}
\label{EqHMSquaredIsId}
HM_{x,E_x/l}HM_{x,l}E=E\otimes O(x)\sim E.
\end{equation} It follows that $HM_{t_i}^2$ is identity. Also Hecke modifications at two different points commute.

Since $\det(HM_{x,l}E)=\det E\otimes O(x)$, Hecke modification interchanges $\Bun_0$ and $\Bun_1$. We see that $S_i=HM_{t_i}$ is an involution of $\Bun_G$ that interchanges $\Bun_0$ and $\Bun_1$. We will use $S_0=HM_0$ to identify $\Bun_0$ and $\Bun_1$. We will write $S_0$ as $HM_0^-=HM_0\otimes O(x)^{-1}$: it is defined similarly to $HM_0$, except the sections are regular and their fiber at $x$ belongs to $l_0$.

In order to write Hecke operators, first we should describe the action of a Hecke modification on a parabolic bundle. Let $2x\in X$ be a point, $l=\begin{pmatrix}
s \\ 1
\end{pmatrix}$ be a line in $E_{2x}$.

\begin{prop}
Let $E$ correspond to $(p,y_1,\ldots,y_m)$. Then $S_0 \circ HM_{2x,l}E$ corresponds to $(q,z_1,\ldots,z_m)$, where $\pm q$ are the roots of
\[\frac{\th(q-x+p)\th(q+x-p)}{\th(q-x-p)\th(q+x+p)}-s\]  and \[z_i=\frac{\tilde{\th}(t_i,-p+x+q)y_i-\tilde{\th}(t_i,p+x+q)}{\tilde{\th}(t_i,-p+x-q)y_i-\tilde{\th}(t_i,p+x-q)},\] where $\tilde{\th}(z,a)=\frac{\th(z-a)}{\th(-a)}$.
\end{prop}
\begin{rem}
Note that these formulas are indeed well-defined with respect to the action $C_2\ltimes \ZN^2$. For example, if we take $-q$ instead of $q$ we get $z_i^{-1}$ instead of $z_i$ and similarly with everything else.
\end{rem}
\begin{proof}
Note that $\begin{pmatrix}
f_p\\0
\end{pmatrix}$ and $\begin{pmatrix}
0\\f_{-p}
\end{pmatrix}$ are the only (up to scaling) rational sections of $\mc{L}_p\oplus\mc{L}_{-p}$ that have a simple pole at zero and vanish at some point. We will find sections $g_1,g_2$ of $S_0\circ HM_{2x,l}E=HM_{0,l_0}^{-}\circ HM_{2x,l}E$ with the same property. Note that a section of $HM_{0,l_0}^{-}\circ HM_{2x,l}E$ with a simple pole at zero is the same as a section of $HM_{0,l_0}\circ HM_{2x,l}E$. Hence $g_1,g_2$ can have a simple pole at zero such that the residue belongs to $l_0$ and a simple pole at $2x$ such that the residue belongs to $l$.

Let $g=\begin{pmatrix}
g_+\\g_-
\end{pmatrix}$. Assume that $g$ vanishes at a point $r$. Then $g_+$ is a section of $\mc{L}_p$ with simple poles at $0,2x$ and a root at $r$, hence it has another root $z_+$ such that $[r]+[z_+]-[0]-[2x]$ is equivalent to $[p]-[0]$. This means $z_+=2x+p-r$. Similarly $g_-$ has the second root $z_-=2x-p-r$. Hence we can assume that for some complex number $c$ \[g_+(z)=c\frac{\th(z-r)\th(z+r-2x-p)}{\th(z)\th(z-2x)},\]\[g_-(z)=\frac{\th(z-r)\th(z+r-2x+p)}{\th(z)\th(z-2x)}.\] Computing the residue at zero gives $c\th(r-2x-p)=\th(r-2x+p)$, hence 
\begin{equation}
\label{EqCThroughYXP}
c=\frac{\th(r-2x+p)}{\th(r-2x-p)}.
\end{equation}
 Computing the residue at $2x$ gives $c\th(r-p)=s\th(r+p)$. Using~\eqref{EqCThroughYXP} we get \[\frac{\th(r-2x+p)\th(r-p)}{\th(r-2x-p)\th(r+p)}=s.\] The left-hand side is an elliptic function in $r$ with two simple poles, hence this equation has two solutions $r_1,r_2$. Setting $r=r_1,r_2$ in formulas for $g_+,g_-,c$ gives sections $g_1,g_2$. Note that $r_1+r_2=2x$.
 
 The section $g_i$ has divisor of zeroes $[r_i]-[0]$, hence it generates a subbundle $\mc{L}_{r_i}$ inside $S_0 \circ HM_{2x,l}E$. It follows that $S_0\circ HM_{2x,l}E$ is isomorphic to $\mc{L}_{r_1}\oplus \mc{L}_{r_2}$ and the isomorphism $\phi\colon \mc{L}_{r_1}\oplus \mc{L}_{r_2}\to S_0 \circ HM_{2x,l}E$ is given by \[\phi\begin{pmatrix}
 \alpha f_{r_1}\\
 \beta f_{r_2}
 \end{pmatrix}=\alpha g_1+\beta g_2=\begin{pmatrix}
 \alpha g_{1+}+\beta g_{2+}\\
 \alpha g_{1-}+\beta g_{2-}
 \end{pmatrix}.\] If we denote $\alpha f_{r_1}$ by $A$ and $\beta f_{r_2}$ by $B$, this becomes \[\phi\begin{pmatrix}
 A\\B
 \end{pmatrix}=\begin{pmatrix}
 A\frac{g_{1+}}{f_{r_1}}+B \frac{g_{2+}}{f_{r_2}}\\
 A\frac{g_{1-}}{f_{r_1}}+B\frac{g_{2-}}{f_{r_2}}
 \end{pmatrix},\] so that $\phi$ is given by the matrix $M(z)=\begin{pmatrix}
 \frac{g_{1+}}{f_{r_1}} & \frac{g_{2+}}{f_{r_2}}\\
 \frac{g_{1-}}{f_{r_1}} & \frac{g_{2-}}{f_{r_2}}
 \end{pmatrix}$. Note that for $i=1,2$, $\eps=\pm$, \[g_{i\eps}=c_{i\eps}\frac{\th(z-r_i)\th(z+r_i-2x-\eps p)}{\th(z)\th(z-2x)}=c_{i\eps}f_{r_i}\frac{\th(z+r_i-2x-\eps p)}{\th(z-2x)},\] where $c_{i\eps}$ is $c_i$ for $\eps=+$ and $1$ for $\eps=-$. Therefore
 \[M(z)=\begin{pmatrix}
 c_1\frac{\th(z+r_1-2x-p)}{\th(z-2x)} & c_2\frac{\th(z+r_2-2x-p)}{\th(z-2x)}\\
 \frac{\th(z+r_1-2x+p)}{\th(z-2x)} & \frac{\th(z+r_2-2x+p)}{\th(z-2x)}
 \end{pmatrix}.\]
 The parabolic structure on $\mc{L}_{r_1}\oplus\mc{L}_{r_2}$ at zero by construction is given by the kernel of \[M(0)\colon (\mc{L}_{r_1}\oplus \mc{L}_{r_2})_0\to E_0.\]  It is generated by \[\begin{pmatrix}
 \th (r_2-2x+p)\\
 -\th(t_1-2x+p)
 \end{pmatrix}.\] Hence we should apply an automorphism of $\mc{L}_{r_1}\oplus\mc{L}_{r_2}$ that divides the first column of $M$ by $-\th(t_1-2x+p)$ and the second by $\th(r_2-2x+p)$. Since we need only the corresponding projective transformation, we can get rid of $\th(z-2x)$ to get
 \[M(z)=\begin{pmatrix}
 \frac{-\th(z+r_1-2x-p)}{\th(r_1-2x-p)} & \frac{\th(z+r_2-2x-p)}{\th(r_2-2x-p)}\\
 \frac{-\th(z+r_1-2x+p)}{\th(r_1-2x+p)} & \frac{\th(z+r_2-2x+p)}{\th(r_2-2x+p)}
 \end{pmatrix}.\]
 
 Note that $r_1+r_2=2x$. Hence the elliptic parameter is $r_1-x=q$, so that $r_2-x=-q$. We can rewrite $M(z)$ as
 \[\begin{pmatrix}
 \frac{-\th(z+q-x-p)}{\th(q-x-p)} & \frac{\th(z-q-x-p)}{\th(-q-x-p)}\\
 \frac{-\th(z+q-x+p)}{\th(q-x+p)} & \frac{\th(z-q-x+p)}{\th(-q+x-p)}
 \end{pmatrix}.\]
It is convenient to introduce $\tilde{\th}(z,a)=\frac{\th(z-a)}{\th(-a)}$, this will give
\[M(z)=\begin{pmatrix}
-\tilde{\th}(z,p+x-q) & \tilde{\th}(z,p+x+q)\\
-\tilde{\th}(z,-p+x-q) & \tilde{\th}(z,-p+x+q)
\end{pmatrix}.\] In order to write $z_i$ in terms of $y_i$ we should use the inverse projective transformation
\[M(z)^{-1}=\begin{pmatrix}
\tilde{\th}(z,-p+x+q) & -\tilde{\th}(z,p+x+q)\\
\tilde{\th}(z,-p+x-q) & -\tilde{\th}(z,p+x-q)
\end{pmatrix},\] this gives
\begin{equation}
\label{EqZiThroughYi}
z_i=\frac{\tilde{\th}(t_i,-p+x+q)y_i-\tilde{\th}(t_i,p+x+q)}{\tilde{\th}(t_i,-p+x-q)y_i-\tilde{\th}(t_i,p+x-q)}.
\end{equation}
\end{proof}
\section{Computation of the Hecke operator}
\subsection{General considerations.}
Sometimes instead of $p,\mb{y}=(y_1,\ldots,y_m)$ we will write $\mb{y}=(y_0,y_1,\ldots,y_m)$, where $y_0=p$, similarly with $z$.

The general definition of Hecke operators is given in~\cite{EFK2}, after Theorem 1.1. It uses the Hecke correspondence \[Z\subset \Bun_G\times \Bun_G\times (X\setminus \{t_0,\ldots,t_m\})\] and projection maps $q_1,q_2,q_3$ to $\Bun_G$, $\Bun_G$, $X$ respectively.

The definition of Hecke operators in~\cite{EFK2} uses an isomorphism \[\mathbf{a}\colon q_1^*(K_{\Bun})\to q_2^*(K_{\Bun})\otimes K_2^2\otimes q_3^*(K_X^{-1}),\] where $K$ denotes the canonical bundle, $K_2$ is the relative canonical bundle of $q_2\times q_3$. In our case $K_X$ is trivial, hence $q_3^*(K_X^{-1})$ is also trivial. 

\begin{rem}
In the notation of~\cite{EFK3}, $\mathbf{a}=a^2$.
\end{rem}

Hecke operators act on the space of square-integrable half-densities, it is defined as follows. Let $Y$ be a variety, $\mc{L}$ be a line bundle on $Y$. Consider the $\RN_{>0}$-principal bundle $\abs{\mc{L}}$. This is a smooth bundle obtained by applying $\abs{\cdot}$ to all transition functions. We also can think of it as a smooth complex line bundle. Similarly we can define any power of $\abs{\mc{L}}$. Set $\Omega_Y=\abs{K_Y}^2$, where $K_Y$ is the canonical bundle of $Y$. We say that $\Omega_Y$ is the bundle of {\it densities}. This means that for any section $f$ of $\Omega_Y$ the notion of being integrable and if so, the integral $\int_Y f$ are well-defined. Thus for any {\it half-density}, i.e. a section of $\abs{K_Y}=\Omega_Y^{\tfrac12}$, the number (possibly, infinite) $\norm{f}^2_{L^2}=\int_{Y}f\ovl{f}$ is well-defined. We define $L^2(Y)$ to be the set of all measurable sections $f$ of $\abs{K_Y}$ such that $\norm{f}_{L^2}$ is finite, modulo the equivalence relation $f\sim g$ if $f$ and $g$ coincide almost everywhere.

Taking $Y=\Bun_G$ we get the Hilbert space $\mc{H}=L^2(\Bun_G)$.

 We can take absolute value of $\mathbf{a}$ to get
\[\abs{\mathbf{a}}\colon q_1^*(\Omega_{\Bun}^{\tfrac12})\to q_2^*(\Omega_{\Bun}^{\tfrac12})\otimes \Omega_2.\] Here $\Omega_2=\abs{K_2}^2$, a relative bundle of densities of $q_2\times q_3$. For $x\in X$ let $\mathbf{a}_x$ be the restriction of $\mathbf{a}$ to $q_3^{-1}(x)$. Now we can take a half-density $f\in\mc{H}$, consider $q_1^*f$, send it to $\abs{\mathbf{a}_x}(q_1^*f)$, and then integrate over fibers of $q_2$ using invariant measure given by the $\Omega_2$ piece to obtain another half-density on $\Bun$. Denote the result by $H_x(f)$. A priori it is not clear whether the integral converges or gives a square-integrable half-density. We will show that $H_x(f)$ is well-defined.

More precisely, consider the subset of nonnegative square-integrable half-densities $S\subset \mc{H}$. Note that for any $f\in \mc{H}$ there exist $f_+,f_-\in S$ such that $f=f_+-f_-$ and $\norm{f}^2=\norm{f_+}^2+\norm{f_-}^2$. Since the integration measure for $H_x$ is positive, $H_x(f)$ can be defined for $f\in S$ as a nonnegative half-density that can be equal to $+\infty$ at some points. It is proved in subsection~\ref{SubSecBounded} that $H_x(f)$ has finite $L^2$-norm. In particular, $H_x(f)$ is finite almost everywhere. This proves that the integral defining $H_x(f)$ converges almost everywhere for $f\in S$. For $f\in \mc{H}$ we use $f=f_+-f_-$ to prove that $H_x(f)$ absolutely converges almost everywhere.

\begin{rem}
\begin{enumerate}
\item
Note that~\cite{EFK3} considers all local fields $F$ and uses the norm map $\|\cdot\|\colon F\to \RN_{>0}$ for uniform treatment of all cases. We work with the field of complex numbers, in this case $\|z\|=\abs{z}^2$. Hence all expressions of the form $\abs{z}^s$ in this paper correspond to $\norm{z}^{\frac{s}{2}}$ in~\cite{EFK3}.
\item
Our plan for defining $H_x$ rigorously is different from~\cite{EFK3}: it proves that there exists a dense subspace $V\subset\mc{H}$ such that the integral defining $H_x$ converges and the result is a bounded operator on $V$, hence it extends from $V$ to a bounded operator on $\mc{H}$.
\end{enumerate}

\end{rem}

Let us fix a point $x$ and restrict $\mathbf{a}$ to $q_3^{-1}(x)$. Denote the restriction by $\mathbf{a}_x$ and $q_3^{-1}(x)$ by $Z_x$. Sometimes we will write $\mathbf{a}$ and $Z$ instead of $\mathbf{a}_x$, $Z_x$. Since $Z_x$ is projective, any nontrivial map from $q_1^*(K_{\Bun})$ to $q_2^*(K_{\Bun})\otimes K_2^2$ is $\mathbf{a}_x$ times a number. We will produce such a map.

Tensoring the source and the target of $\mathbf{a}$ by $K_Z^{-1}$ we get a map \[\phi_x\colon T_1\to T_2\otimes K_2^2,\] where the fiber of $T_1$ is the tangent space to $q_1^{-1}(\mby)$, and similarly for $T_2$ and $q_2$. 


Fix $\mbz$. We note that $(q_2\times q_3)^{-1}(\mbz,x)$ is isomorphic to $\Pone$.  Hence $T_2$ becomes the tangent bundle on $\Pone$ and $K_2$ is by definition the canonical bundle on $(q_2\times q_3)^{-1}(\mbz,x)$. Denote the restriction of the isomorphism $\phi_x$ above by $\phi_{\mbz,x}\colon T_{1,\mbz,x}\to T_{2,\mbz,x}\otimes K_2^2$.  



Let us modify $Z_x$ so that it consists of triples $(E,F,l)$, where $l\in \mathbb{P}(F_x)\cong \Pone$ such that $E=HM_{x,l}F$. Then we have a map $\pi\colon Z_x\to \Pone$. Note that $\mathbb{P}(F_x)=q_2^{-1}(F)$.


Restricting $d\pi$ to $T_1$ we get a map $d\pi\colon T_1\to T_2$. Suppose that the restriction of this map to $q_2^{-1}(F)$ has $k$ zeroes at points $s_1,\ldots,s_k\in\Pone$. Realize $O(-k)$ as a bundle whose regular sections are regular functions with zeroes at $s_1,\ldots,s_k$. Then we have a rational map $\gamma$ of line bundles from $O$ to $O(-k)$ that sends a function to itself. Hence $d\pi\otimes \gamma$ is an isomorphism between $T_1$ and $T_2\otimes O(-k)$. Comparing this to the formula for $\phi_{\mbz,x}$ above and noticing that $K_{\Pone}^2=O(-4)$, we get that $k=4$. So we have to find four zeroes of $d\pi$.

Suppose that $E=E_{\mb{y}}$, $F=F_{\mb{z}^0}$, $l$ corresponds to $v_0\in\CN$. Let us write a local analytic parametrization of $q_1^{-1}(\mby,x)$ at $(\mby,\mbz,v)$: $(\mby,\mbz(t),v(t))$.  Here $t\in (-1,1)$, $v(0)=v^0$, $\mby=\psi(\mbz(t),v(t))$, $\mbz(0)=\mbz^0$. The fiber $T_{1,\mby,\mbz}$ of $T_1$ over $(\mby,\mbz)$ is generated by \[(0,\partial_t\mbz(0),\partial_t v(0))\subset T_{\mby} \Bun\times T_{\mbz} \Bun\times T_{\Pone}.\] 
Hence at the points where $\partial_t v(0)$ is nonzero, the map $d\pi$ is an isomorphism. At these points let $\mbz'_v:=\frac{\mbz'_t}{v'_t}$, where $\mbz'_t=\partial_t\mbz$, $v'_t=\partial_t v$. We can write $\mby=\psi(\mbz(v(t)),v(t))$. Since $\mby$ is constant, differentiating with respect to $t$ and setting $t=0$, we get \[v'_t(0)\cdot\big(\dd{z}\psi(\mbz(v_0),v_0)\mbz'_v(v_0)+\partial_v\psi(\mbz(v_0),v_0)\big)=0.\] Here by $\dd{z}\psi$ we mean the square matrix of partial derivatives. Since we assumed that $v'_t\neq 0$, we get that \[\mbz'_v(v_0)=(\dd{z}\psi(\mbz,v_0))^{-1}\partial_v \psi(\mbz,v_0).\] 

It follows that any value of $v_0$ for which at least one of the components of \[(\dd{\mbz}\psi(\mbz(v_0),v_0))^{-1}\partial_v\psi(\mbz(v_0),v_0)\] has a pole gives a zero of $d\pi$. We will find four such points, giving us four zeroes of $d\pi$. 
\subsection{Computing the poles of $d\pi$}
Let us interchange $y$ and $z$, so that we are finding the poles of the components of \[(\dd{\mby}\psi(\mby(v_0),v_0))^{-1}\dd{v}\psi(\mby(v_0),v_0),\] where $\mb{z}=\psi(\mb{y},v)$. We have $\psi(p,y_1,\ldots,y_m,v)=(q,z_1,\ldots,z_m)$, where $q$ depends on $p,v$ and $z_i$ depend on $y_i,p,v$. Hence \[\dd{\mby}\psi(\mby(v),v)=\begin{pmatrix}
q'_p & 0 & \ldots & 0\\
(z_1)'_p & (z_1)'_{y_1} & \ldots & 0\\
\vdots & \vdots & \ddots & \vdots\\
(z_m)'_p & 0 & \ldots & (z_m)'_{y_m}
\end{pmatrix},\] so that 
\[(\dd{\mby}\psi(\mby(v_0),v_0))^{-1}\dd{v}\psi(\mby(v_0),v_0)=\begin{pmatrix}
\frac{q'_v}{q'_p}\\
\frac{(z_1)'_v-\frac{q'_v}{q'_p}(z_1)'_p}{(z_1)'_{y_1}}\\
\vdots\\
\frac{(z_m)'_v-\frac{q'_v}{q'_p}(z_m)'_p}{(z_m)'_{y_m}}
\end{pmatrix}.\]

We have $\frac{\th(p+q-x)\th(q+x-p)}{\th(q-x-p)\th(q+x+p)}=v$. Denote the left-hand side by $f(p,q,x)$. We have $f'_qq'_v=1$, so that $q'_v=\frac{1}{f'_q}$ and $f'_qq'_p+f'_p=0$, so that $q'_p=-\frac{f'_p}{f'_q}$. We get $\frac{q'_v}{q'_p}=-\frac{1}{f'_p}$. The poles of $-\frac{1}{f'_p}$ are the roots of $f'_p$. The function $f'_p$ is an elliptic function in $q$ with double poles at $q=\pm(x+p)$. Hence there exists a polynomial in $v$ of degree two, $P(v)=P(v,p,x)$, such that $f'_p=P(f)$. It follows that $P(v)$ has roots at exactly two points corresponding to the poles of $\frac{q'_v}{q'_p}$.

Now we turn to the poles corresponding to $z_i$, where $i$ is an integer from $1$ to $m$. Let $z_i:=z_i(p,q(p,v),y_i)$. Then $(z_i)'_{p,\text{old}}=(z_i)'_{p,\text{new}}+(z_i)'_qq'_p$ and $(z_i)'_{v,\text{old}}=(z_i)'_qq'_v$. Then \[\frac{q'_v}{q'_p}(z_i)'_{p,\text{old}}=\frac{(z_i)'_{p,\text{new}}q'_v}{q'_p}+(z_i)'_qq'_v.\] Hence \[(z_i)'_{v,\text{old}}-\frac{q'_v}{q'_p}(z_i)'_{p,\text{old}}=-\frac{(z_i)'_{p,\text{new}}q'_v}{q'_p}.\]

The function $z_i$ is given by equation~\eqref{EqZiThroughYi}: \[z_i=\frac{\tilde{\th}(t_i,-p+x+q)y_i-\tilde{\th}(t_i,p+x+q)}{\tilde{\th}(t_i,-p+x-q)y_i-\tilde{\th}(t_i,p+x-q)}.\]
For any function $z=\frac{A(p)y+B(p)}{C(p)y+D(p)}$ we have \[z'_p=\frac{(A'y+B')(Cy+D)-(Ay+B)(C'y+D')}{(Cy+D)^2}\] and \[z'_y=\frac{AD-BC}{(Cy+D)^2},\] hence \[\frac{z'_p}{z'_y}=\frac{(A'y+B')(Cy+D)-(Ay+B)(C'y+D')}{AD-BC}.\] The poles of this expression come from the poles of $A,B,C,D$ and the zeroes of $AD-BC$.

Take $A,B,C,D$ as above. The poles of $A,B,C,D$ are $q=\pm p\pm x$, this gives $v=0,\infty$. As a function of $t_i$, $AD-BC$ has zeroes at $t_i=0,2x$ and no poles, hence $AD-BC=L(p,q,x)\th(t_i)\th(t_i-2x)$. At $t_i=p+x+q$ we have \[\tilde{\th}(t_i,\eps_1 p+x+\eps_2 q)=\frac{\th(t_i-\eps_1p-x-\eps_2q)}{\th(-\eps_1p-x-\eps_2q)}=\frac{\th((1-\eps_1)p+(1-\eps_2)q)}{\th(-\eps_1p-x-\eps_2q)},\] where $\eps_1=\pm 1$ and $\eps_2=\pm 1$. Since $B(p+x+q)=0$, we have 
\[\frac{-\th(2p)\th(2q)}{\th(p-x-q)\th(q-x-p)}=L(p,q,x)\th(p+q+x)\th(p+q-x),\]
so that \[L(p,q,x)=\frac{-\th(2p)\th(2q)}{\th(p-x-q)\th(q-x-p)\th(p+q+x)\th(p+q-x)}.\]

Hence the zeroes of $AD-BC$ correspond to $q$ such that $2q=0$. The corresponding bundle has group of automorphisms of dimension at least three. Since $\mc{L}_p\oplus\mc{L}_{-p}$ has the group of automorphisms $\CN^{\times}\times \CN^{\times}$, the set of bundles with $2q=0$ has measure zero.

It follows that $\mathbf{a}$ is given by $d\pi\otimes \frac{(dv)^2}{vP(v)}$. Hence the integration measure is $\frac{dv\ovl{dv}}{\abs{vP(v)}}$. It can also be shown that the remaining piece is the square root of Jacobian, namely $\abs{\det\dd{\mb{y}}\mb{z}}$. In the expression $\mb{z}=\mb{z}(\mb{y},v)$ the components on $\mb{z}$ depend on the components of $\mb{y}$ as follows: $q$ depends on $p,v$, and $z_i$ depends on $y_i,p,v$. Hence the determinant equals to the product of the diagonal elements: \[\det\dd{\mb{y}}{\mb{z}}=q'_p\cdot\prod_{i=1}^m \frac{\partial}{\partial y_i} z_i(p,q,y_i).\] Taking out $q'_p$, we can write 
\begin{equation}
\label{EqHeckeOperatorBeforeCoordinateChange}
H_xf(\mb{y})=\intl_{\C}\prod_{i=1}^m \big|\frac{\partial}{\partial y_i} z_i(p,q,y_i)\big|f\big(q(p,v),\mb{z}\big)\frac{\abs{q'_p}dv\ovl{dv}}{\abs{vP(v)}},
\end{equation} where $z_i$ is given by~\eqref{EqZiThroughYi}.
\subsection{Change of variable $v=f(q)$.}
In this subsection we will rewrite the equation~\eqref{EqHeckeOperatorBeforeCoordinateChange} in a more convenient form. Let us introduce some notation. Let $K(p,q,x)=\mathbf{K}(\wp(p),\wp(q),\wp(x))$, where $\mathbf{K}$ is the kernel for Hecke operator on $\Pone$ with four parabolic points given by the formula in~\cite{EFK3}, Proposition 5.5 in the case $F=\C$. Let $\mb{y}=(y_1,\ldots,y_m)$. Let $\mb{z}=(z_1,\ldots,z_m)$ be given by~\eqref{EqZiThroughYi}. Suppose that $g_{x,p,q,i}$ is an element of $\PGL_2$ corresponding to the projective transformation $z_i(y_i)$, that is,
 \[g_{x,p,q,i}=\begin{pmatrix}
\tilde{\th}(t_i,-p+x+q) & -\tilde{\th}(t_i,p+x+q)\\
\tilde{\th}(t_i,-p+x-q) & -\tilde{\th}(t_i,p+x-q)
\end{pmatrix},\] here $\tilde{\th}(z,a)=\frac{\th(z-a)}{\th(-a)}$. Let $g_{x,p,q}=(g_{x,p,q,1},\cdots,g_{x,p,q,m})$ be an element of $\PGL_2^m$, $\rho$ be the standard representation of $\PGL_2^m$ on $L^2(\Pone)^{\otimes m}$. Then we have the following theorem:
\begin{thr}
\label{ThrGroupAndKernelHeckeOperators}
 Then the Hecke operator can be written as \begin{multline}
\label{EqGroupAndKernelHeckeOperator}
(H_xf)(p,\mb{y})=\intl_{\C}\prod_{i=1}^m \big|\frac{\partial}{\partial y_i} z_i(p,q,y_i)\big|f\big(q,\mb{z}\big)K(p,q,x)dq=\\
=\intl_{\C}\big(\rho(g_{x,p,q})f\big)(q,\mb{y})K(p,q,x)dq.
\end{multline} In particular, for $m=0$, $H_x$ is unitarily equivalent to the Hecke operator on $\Pone$ with four parabolic points.
\end{thr}

\begin{proof}

Recall that $v=f(p,q,x)$. The function $f(q):=f(p,q,x)$ is an elliptic function with two poles, hence $q\mapsto f(q)$ is a two-to-one surjective map. We will make the change of variable $v=f(q)$ in the integral~\eqref{EqHeckeOperatorBeforeCoordinateChange}. The measure \[\frac{\abs{q'_p}dv\ovl{dv}}{\abs{vP(v)}}\] becomes
\[\frac{\abs{q'_p}\abs{f'(q)}^2dqd\ovl{q}}{\abs{ff'_p}}.\] Note that $q'_p=\frac{-f'_p}{f'_q}$, so that the measure is 
\[\bigg|\frac{f'_q}{f}\bigg|dq\ovl{dq}.\]
Recall that $f(p,q,x)=\frac{\th(p+q-x)\th(q+x-p)}{\th(q-x-p)\th(q+x+p))}$. Let $Z(t)=(\log\th(t))'$, this is an odd function in $t$ satisfying $Z(t+1)=Z(t)$, $Z(t+\tau)=Z(t)-2\pi i$. Then \[(\log f)'_q=Z(p+q-x)+Z(q+x-p)-Z(q-x-p)-Z(q+x+p).\] Note that $(\log f)'_q$ is an odd function of $q$ with poles at $q=\pm p\pm x$. Hence we can write \[(\log f)'_q=\frac{C\wp'(q)}{(\wp(q)-\wp(p-x))(\wp(q)-\wp(p+x))},\] since both sides have the same set of zeroes and poles. Here $C$ does not depend on $q$. Note that the residues on both sides at $q=p-x$ are $1$ and $\frac{C\wp'(p-x)}{\wp'(p-x)(\wp(p-x)-\wp(p+x))}$, hence $C=\wp(p-x)-\wp(p+x)$.

Note that \begin{multline*}
(\log f)'_q=Z(p+q-x)+Z(q+x-p)-Z(q-x-p)-Z(q+x+p)=\\
Z(p+q-x)+Z(q+x-p)+Z(x+p-q)-Z(q+x+p)
\end{multline*}
is symmetric in $p,q,x$. Let us express \[\frac{\wp(p-x)-\wp(p+x)}{(\wp(q)-\wp(p-x))(\wp(q)-\wp(p+x))}\] as $\wp'(p)\wp'(x)$ over a polynomial in $\wp(p)$, $\wp(x)$. Let $r=\wp(q)$, $s=\wp(p)$, $t=\wp(x)$.

Counting zeroes and poles we have \[\wp(p-x)-\wp(p+x)=\frac{\wp'(p)\wp'(x)}{(\wp(p)-\wp(x))^2}.\] Hence we are interested in \begin{multline*}(\wp(p)-\wp(x))^2(\wp(q)-\wp(p-x))(\wp(q)-\wp(p+x))=\\
(s-t)^2(r-\wp(p-x))(r-\wp(p+x)).\end{multline*} The coefficient of $r^2$ is $(s-t)^2$. 

The coefficient of $r$ is \begin{equation}
\label{EqCoeffOnR}
(\wp(p-x)+\wp(p+x))(\wp(p)-\wp(x))^2=A\wp(p)^2+B\wp(p)+C
\end{equation} for some $A,B,C$. Setting $p=0$ we get $A=2\wp(x)$. We can compute more terms of the Taylor expansion at zero: \[\wp(p-x)+\wp(p+x)=2\wp(x)+p^2\wp''(x)+O(p^3)\] \[\wp(p)-\wp(x)=p^{-2}-\wp(x)+O(p^2).\] Hence \[(\wp(p-x)+\wp(p+x))(\wp(p)-\wp(x))^2=2\wp(x)p^{-4}+(\wp''(x)-4\wp(x)^2)p^{-2}+O(p^{-1}).\] It follows that $B=\wp''(x)-4\wp(x)^2$. Since $\wp'(x)^2=4\wp(x)^3-g_2\wp(x)-g_3$, we get $\wp''(x)=6\wp(x)^2-\frac{g_2}{2}$, so that $B=2\wp(x)^2-\frac{g_2}{2}$. Setting $p=x$ in~\eqref{EqCoeffOnR} gives us \[\wp'(x)^2=A\wp(x)^2+B\wp(x)+C,\] hence \begin{multline*}
C=\wp'(x)^2-A\wp(x)^2-B\wp(x)=\\
4\wp(x)^3-g_2\wp(x)-g_3-2\wp(x)^3-2\wp(x)^3+\frac{g_2}{2}\wp(x)=\\
\frac{-g_2}{2}\wp(x)-g_3.
\end{multline*}

It follows that \begin{multline}
\label{EqSecondExpressionForPolynomial}
(\wp(p)-\wp(x))^2(\wp(q)-\wp(p-x))(\wp(q)-\wp(p+x))=\\
(s-t)^2r^2-(2ts^2+2t^2s-\frac{g_2}{2}(t+s)-g_3)r+(\wp(p)-\wp(x))^2\wp(p-x)\wp(p+x).
\end{multline} We know that this expression is symmetric in $r,s,t$, hence the constant coefficient is $s^2t^2+\frac{g_2}{2}st+g_3(s+t)+C_0$, where $C_0$ is a number.

Using~\eqref{EqSecondExpressionForPolynomial} we get \[(\wp(p)-\wp(x))^2\wp(p-x)\wp(p+x)=\wp(x)^2\wp(p)^2+(\frac{g_2}{2}\wp(x)+g_3)\wp(p)+g_3\wp(x)+C_0.\] Setting $p=x$ we obtain \[\wp'(x)^2\wp(2x)=\wp(x)^4+\frac{g_2}{2}\wp(x)^2+2g_3\wp(x)+C_0.\]  We have the duplication formula \[\wp(2x)=\frac{1}{4}\bigg(\frac{\wp''(x)}{\wp'(x)}\bigg)^2-2\wp(x),\] so that \begin{multline*}
\wp'(x)^2\wp(2x)=\frac{1}{4}\wp''(x)^2-2\wp'(x)^2\wp(x)=\\
\frac{1}{4}\big(6\wp(x)^2-\frac{g_2}{2}\big)^2-2\big(4\wp(x)^3-g_2\wp(x)-g_3\big)\wp(x)=\\\
\wp(x)^4+\frac{g_2}{2}\wp(x)+2g_3\wp(x)+\frac{g_2^2}{16}
\end{multline*}
It follows that $C_0=\frac{g_2^2}{16}$. So the polynomial is
\[ r^2s^2+r^2t^2+s^2t^2-2rst(r+s+t)+\frac{g_2}{2}(rs+st+rt)+g_3(r+s+t)+\frac{g_2^2}{16}.\]

It can be checked that after a shift this polynomial coincides with the polynomial from~\cite{EFK3}, section 5 and~\cite{Ko}, page 3.
Combining this computation with~\eqref{EqHeckeOperatorBeforeCoordinateChange} finishes the proof.
\end{proof}
\section{Basic properties of Hecke operators}
\subsection{Hecke operators are bounded}
\label{SubSecBounded}
Boundedness of the Hecke operator in the case of one marked point ($m=0$) follows from unitary equivalence with Hecke operators for four marked points on $\Pone$ which is shown in~\cite{EFK3}, Section 5. The case $m>0$ is done as follows. We have \[H_x(f)(p,y_1,\ldots,y_m)=\int_{X}U_{p,q}f(q)(\mb{y})K(p,q,x)dq,\] where $U_{p,q}$ is a unitary operator in the last $m$ variables corresponding to the coordinate change~\eqref{EqZiThroughYi} and $K$ is the kernel for the Hecke operator in the case $m=0$. We have
\begin{multline*}
(H_x(f),g)=\int_{p,\mb{y}}H_xf(p,\mb{y})g(p,\mb{y})dp d\mb{y}=\\
\int_{p,q,\mb{y}}U_{p,q}f(q)(\mb{y})g(p,\mb{y})K(p,q,x)dp dq d\mb{y}\leq\\
\int_{p,q}\|U_{p,q}f(q)\|_{L^2}\|g(p)\|_{L^2}K(p,q,x)dp dq.
\end{multline*}
Here by $\|\cdot\|_{L^2}$ we mean $L^2$ norm in the last $m$ variables. We used the Cauchy-Schwarz inequality in the last line. We have \[\norm{U_{p,q}f(q)}_{L^2}=\norm{f(q)}_{L^2}=:F(q).\] The function $F$ belongs to $L^2(X)$ and satisfies $\norm{F}=\norm{f}$. Similarly, $\norm{g(p)}_{L^2}=G(p)\in L^2(X)$, and $\norm{G}=\norm{g}$. Then \[(H_x(f),g)\leq (H^0_x(F),G)\leq\norm{H^0_x}\norm{F}\norm{G}=\norm{H^0_x}\norm{f}\norm{g},\] where $H^0_x$ is the Hecke operator for $m=0$. Hence $H_x$ is bounded with $\norm{H_x}\leq\norm{H^0_x}$.
\subsection{Hecke operators are self-adjoint and commute with each other}
We have \[H_x^*f(p,y_1,\ldots,y_m)=\int_{X}U_{q,p}^{-1}f(q)K(q,p,x)dq.\] Using~\eqref{EqZiThroughYi} we see that $U_{q,p}^{-1}=U_{p,q}$. Since $K(q,p,x)=K(p,q,x)$, we have $H_x^*=H_x$.

Let $x,y$ be points on a curve. For any $E=\mc{L}_p\oplus\mc{L}_{-p}$ and any $l_x\in E_x$, $l_y\in E_y$ we have $HM_{y,l_y}HM_{x,l_x}E\cong HM_{y,l_y}HM_{x,l_x}E$. Suppose that $l_x$, $l_y$ are parametrized by $v,w\in\Pone$. Then $l_y\in (HM_{x,l_x}E)_y$ is parametrized by \[W=\frac{\tilde{\th}(y,-p+x+q)w-\tilde{\th}(y,p+x+q)}{\tilde{\th}(y,-p+x-q)w-\tilde{\th}(y,p+x-q)}.\]
Writing $HM_{y,l_y}HM_{x,l_x}=HM_{x,l_x}HM_{y,l_y}$ in coordinates gives $U_{y,W}U_{x,v}=U_{x,V}U_{y,w}$, where $U$ is the unitary operator corresponding to the Hecke modification and
\[V=\frac{\tilde{\th}(x,-p+y+q)v-\tilde{\th}(x,p+y+q)}{\tilde{\th}(x,-p+y-q)v-\tilde{\th}(x,p+y-q)}.\]
 We have
 \begin{multline*}
 H_xH_yf=\int_{\Pone}U_{x,v}(H_yf)d\mu(v,p)=\\
 \int_{\Pone\times\Pone}U_{x,v}U_{y,w}fd\mu(w,q(v,p))d\mu(v,p)=\\
 \int_{\Pone\times\Pone}U_{x,V}U_{y,w}fd\mu(w,q(v,p))V'_vd\mu(V,p)
 \end{multline*}
 The operator $U_{x,V}U_{y,w}$ does not change after interchanging $x,y$ and $u,v$ simultaneously and it can be checked that same is true for the measure $d\mu(w,q(v,p))V'_vd\mu(V,p)$.
\subsection{Hecke operators are compact}
We know from~\cite{EFK3} that $H_x^0$ is compact and we want to extend this to the case $m>0$. We will do it in several steps.

 Since $H_x^0$ is compact, it can be approximated by finite rank operators $H_x^{0,k}$ with kernel given by a finite sum $K_k=\sum_l g_l(p,x)h_l(q,x)$, where $g_l,h_l$ belong to $L^2$. The $K_k$ can be chosen to be symmetric in $p,q$. Let $H_{x,k}$ be the operator given by $K_k$ instead of $K$ in the formula~\eqref{EqGroupAndKernelHeckeOperator}. Reasoning as in Subsection~\ref{SubSecBounded}, we see that $H_{x,k}$ tends to $H_x$ in norm. Hence it is enough to prove that $H_{x,k}$ is compact. So we can assume that $H_x=H_{x,k}$ has finite rank kernel $K=K_k$ in the formula~\eqref{EqGroupAndKernelHeckeOperator}.

Since $H_x$ is self-adjoint, it is enough to prove that $H_x^n$ is compact for some power $n$.
 We have 
\begin{multline}
\label{EqHxToN}
H_x^n(f)=\int_{p_1,\ldots,p_n}U_{p,p_1}U_{p_1,p_2}\cdots U_{p_{n-1},p_n}f(p_n,-)\\
 K(p,p_1,x)K(p_1,p_2,x)\cdots K(p_{n-1},p_n,x) dp_1\cdots dp_n.
 \end{multline}
Note that each $U_{p_i,p_{i+1}}$ is the action of a group element $g(p_i,p_{i+1})$.
 \begin{lem}
 \label{LemMapIsDominant}
 For large enough $n$ for any $p,p_n\in X$, the map $\phi\colon X^{n-1}\to G$ given by $\phi(p_1,\ldots,p_{n-1})= g(p,p_1)\cdots g(p_{n-1},p_n)$ is dominant.
 \end{lem}
 \begin{proof}
It follows from the Drinfeld-Simpson theorem that any $G$-bundle on a non-compact curve is trivial. Hence for any two sequences $(p,y_1,\ldots,y_m)$, $(q,z_1,\ldots,z_m)$ there exist a number $N$ and a sequence of Hecke modifications $H_{x,l_1},\ldots,H_{x,l_N}$ that sends $E_{p,\mb{y}}$ to $E_{q,\mb{z}}$. Fix any $p_0$. Consider the subgroup $G_0$ of $\PGL_2^m$ generated by $g(p_0,p_1)\cdots g(p_k,p_0)$ for all positive integers $k$ and sequences $(p_1,\ldots,p_k)$. It follows that $G_0$ acts transitively on $(\Pone)^m$. It can be shown that in this case $G_0$ coincides with $\PGL_2^m$. Arguing similarly to~\cite{EFK3}, Lemma 8.9, we get that $\phi$ is dominant for $p=p_n=p_0$. It follows that $\phi$ is dominant for any pair $(p,p_n)$.
 \end{proof}

We will prove that $H_x^n$ is compact for $n$ given by Lemma~\ref{LemMapIsDominant}. Writing $K$ as a sum of rank one kernels in~\eqref{EqHxToN} we get that $H_x^n$ is a sum of operators
\begin{multline*}L(f)=\int_{p_1,\ldots,p_n}U_{p,p_1}\cdots U_{p_{n-1},p_n}f(p_n,-)\\
K_1(p,p_1,x)\cdots K_n(p_{n-1},p_n,x) dp_1\cdots dp_n,
\end{multline*}
where each $K_i=g_ih_i$ has rank one. Taking a limit if necessary we can assume that $g_i,h_i$ satisfy $\eps<\abs{g_i(z)},\abs{h_i(z)}<\eps^{-1}$ for some $\eps>0$ and all $z\in \Pone$.

Hence after the change of variable given by $\phi$ we get
\[L(f)=\int_{G,p_n}g(f)d\mu(g,p,p_n),\]
where $\mu=\phi_*\nu$ and $\nu$ is the measure \[K_1(p,p_1,x)\cdots K_n(p_{n-1},p_n,x) dp_1\cdots dp_{n-1}.\] Since $\nu$ depends continuously on $p,p_n$, $\mu$ also depends continuously on $p,p_n$. Also note that $\int_{G}d\mu(g,p,p_n)$ is the convolution of $K_1,\ldots,K_n$, which we denote by $S(p,p_n)$.

Using Theorem 8.6 and the proof of Proposition 3.13 from~\cite{EFK3} we get the following: \[f\mapsto \int_{G}g(f)d\mu(g,p,p_n)\] is a compact operator $S(p,p_n)R_{p,p_n}$ with norm at most $\abs{S(p,p_n)}$, so that $R$ has norm at most one.  Hence \[L(f)=\int_{p_n}R_{p,p_n}(f(p_n,-))S(p,p_n)dp_n.\] Since $\mu$ continuously depends on $p,p_n$, the operator $S(p,p_n)R_{p,p_n}$ norm-continuously depends on $p,p_n$. Since $\abs{S}$ belongs to the interval $(\eps^{2},\eps^{-2})$, the operator $R_{p,p_n}$ norm-continuously depends on $p,p_n$. 
Using the standard argument for any $\eps_1>0$ we find functions $c_1,c_2\colon \Pone \to \Pone$ with finite image such that $\norm{R_{p,p_n}-R_{c_1(p),c_2(p_n)}}<\eps_1$. We obtain

\begin{lem}
$L$ can be approximated by \[(L_c\phi)(p,y_1,\ldots,y_m)=\int_{p_n}R_{c_1(p),c_2(p_n)}(\phi(p_n,-))K(p,p_n) dp_n.\]
\end{lem}
The functions $c_1,c_2$ define two partitions of $\Pone$ into finite number of pieces $U_i^k$ such that $c_k$ is constant on $U_i^k$. For $p\in U_i^1$, $p_n\in U_j^2$ denote $R_{c_1(p),c_2(p_n)}$ by $R_{ij}$. We have $L_c=\sum L_{ij}$, where \[L_{ij}=\chi_{U_i^1}\int_{U_{j}^2}R_{ij}(\phi(p_n,-))S(p,p_n) dp_n.\] Now each $L_{ij}$ is the tensor product of two compact operators: rank one operator with kernel $\chi_{U_i^1}(p)\chi_{U_j^2}(p_n)S(p,p_n)$ in variable $p$ and $R_{ij}$ in variables $y_1,\ldots,y_m$. Hence $L_{ij}$ is compact, so $L$ is compact.







\end{document}